\input ./style/arxiv-general.cfg
\documentclass[MSNbibl,number,citesort,dvips]{arxbj}
\makeatletter
   \@ifpackageloaded{graphicx}{}{\usepackage{graphicx}}
\makeatother
\usepackage[algoruled]{algorithm2e}
\usepackage{tikz,mparhack,pgfplots}
\usetikzlibrary{pgfplots.groupplots}
\usetikzlibrary{decorations.pathreplacing}
\usepackage{mathrsfs}

\aid{0}
\volume{21}
\issue{4}
\pubyear{2015}
\firstpage{2289}
\lastpage{2307}
\doi{10.3150/14-BEJ644} 
\docsubty{FLA}

\makeatletter

\newcommand{\rrvert}{\vert}
\newcommand{\rrVert}{\Vert}
\newcommand{\llvert}{\vert}
\newcommand{\llVert}{\Vert}

\newcommand{\N}{\mathbb{N}}
\newcommand{\R}{\mathbb{R}}
\renewcommand{\P}{\mathbb{P}}
\newcommand{\E}{\mathbb{E}}

\renewcommand{\theenumi}{\roman{enumi}}

\newcommand{\diam}{\operatorname{diam}}
\newcommand{\argmin}{\operatorname{\arg\min}\limits}
\newcommand{\eqref}[1]{(\ref{#1})}

\newproclaim{definition}{Definition}
\newtheorem{proposition}{Proposition}
\newtheorem{corollary}{Corollary}
\newproclaim{invariant}{Invariant}
\newtheorem{theorem}{Theorem}
\makeatother

\begin{document}
\begin{frontmatter}

\title{Adaptive-treed bandits}
\runtitle{Adaptive-treed bandits}

\begin{aug}
\author[A]{\inits{A.D.}\fnms{Adam D.}~\snm{Bull}\corref{}\ead[label=e1]{a.bull@statslab.cam.ac.uk}}
\address[A]{Statistical Laboratory, Wilberforce Road, Cambridge CB3 0WB,
UK.\\ \printead{e1}}
\end{aug}

\received{\smonth{2} \syear{2013}}
\revised{\smonth{2} \syear{2014}}

%
\begin{abstract}
We describe a novel algorithm for noisy global optimisation and
continuum-armed bandits, with good convergence properties over any
continuous reward function having finitely many polynomial maxima.
Over such functions, our algorithm achieves square-root regret in
bandits, and inverse-square-root error in optimisation, without prior
information.

Our algorithm works by reducing these problems to tree-armed bandits,
and we also provide new results in this setting. We show it is
possible to adaptively combine multiple trees so as to minimise the
regret, and also give near-matching lower bounds on the regret in
terms of the zooming dimension.
\end{abstract}

%
\end{frontmatter}

\section{Introduction}
\label{sec:introduction}

In \emph{noisy global optimisation}, we wish to maximise a continuous
function $\mu\dvtx X \to[0,1]$ over a space $X = [0,1]^p$, given only
noisy observations of the function values $\mu(x)$. This problem
arises in a wide variety of engineering applications, and has been
considered by many authors (for example, see references
in \cite{parsopoulos_recent_2002,huang_sequential_2006,frazier_knowledge-gradient_2009,srinivas_gaussian_2010}).

To be precise, we suppose that at each time $t$, we choose a design
point $x_t \in X$, and then observe a random variable $Y_t \in
[0,1]$ with mean $\mu(x_t)$, as in Figure~\ref{fig:cab}. After $T$
steps, our goal is to choose an estimated maximum $\hat x_T$ of
$\mu$, so as to minimise the \emph{simple regret},
%
\begin{equation}
\label{eq:s-def} S_T = \mu^* - \mu(\hat x_T),
\end{equation}
where $\mu^* = \sup_{x \in X} \mu(x)$.

%
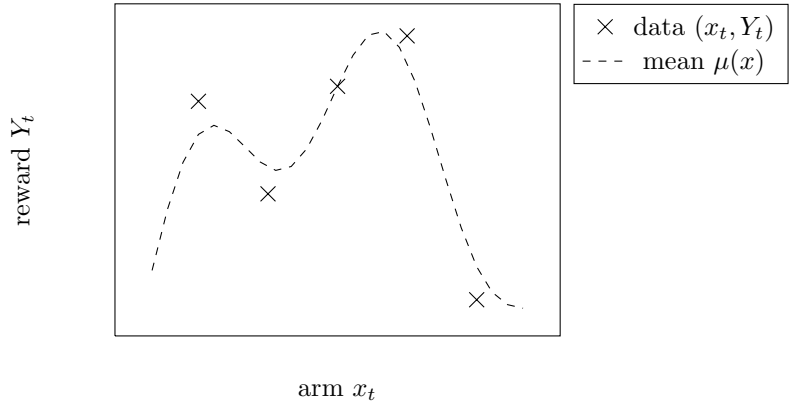
\begin{figure}
\begin{center}
\begin{tikzpicture}
\begin{axis}[ticks=none,height=6cm,width=7.5cm,
xlabel=arm $x_t$,ylabel=reward $Y_t$,
legend style={legend pos=outer north east}]
\addplot[domain=-0.75:0.75,only marks,mark=x,mark
size=4pt,samples=5] {-x^2+sin(6*deg(x))/2};
\addlegendentry{data $(x_t,Y_t)$}
\addplot[domain=-1:1,dashed] {-x^2+sin(6*deg(x))/3};
\addlegendentry{mean $\mu(x)$}
\end{axis}
\end{tikzpicture}
\end{center}
\caption{Noisy global optimisation: we choose design points $x_t$,
observe data $Y_t$ with mean $\mu(x_t)$, and wish to maximise
$\mu$.}
\label{fig:cab}
\end{figure}

We would like to find a solution to this problem which achieves good
rates of convergence, and can also be expected to provide good
practical performance. We note that good convergence of $S_T$ does
not necessarily ensure good practical performance: for example, if
$\mu$ is Lipschitz on $[0,1]$, the optimal rate of $\tilde
{\mathrm{O}}(T^{-1/3})$ can be achieved by a fixed choice of design points
$x_t$; nonetheless, we can expect better practical performance from
a choice which varies with the observations $Y_t$. (The result for a
fixed design is given by \cite{muller_kernel_1985}; the
corresponding lower bound can be proved similarly to our
Theorem~\ref{thm:lower-bound}.)

An alternative is to instead minimise the \emph{cumulative regret},
%
\begin{equation}
\label{eq:r-def} R_T = \sum_{t=1}^T
\bigl(\mu^* - \mu(x_t) \bigr).
\end{equation}
If an algorithm controls the cumulative regret at rate $Tr_T$, it
can also control the simple regret at rate $r_T$
\citep{bubeck_pure_2009}; bounding the cumulative regret is thus a
stronger result. The advantage in bounding $R_T$ is that it also
ensures our solution will place most of its design points in regions
where $\mu$ is near-optimal; that few observations will be wasted.

We would thus expect algorithms which control the cumulative regret to
offer improved practical performance. For example, in our Lipschitz
model above, a fixed choice of design points must suffer $\Omega(T)$
cumulative regret; an algorithm which concentrates its design points
in optimal regions of $\mu$ can simultaneously achieve the optimal
rates of $\tilde {\mathrm{O}}(T^{2/3})$ cumulative regret, and $\tilde
{\mathrm{O}}(T^{-1/3})$ simple regret \citep{kleinberg_nearly_2005}.

In the following, we will therefore seek an algorithm for choosing the
design points $x_t$ which minimises the cumulative regret. Problems
of this kind are known as \emph{multi-armed bandits}; they can be
thought of as attempting to optimally play an unknown slot machine
(or `bandit') with multiple arms.

The field of multi-armed bandits has a long history in the literature,
and comprises many difficult problems even when the set $X$ to
optimise over is small and finite (see references
in \cite{bubeck_regret_2012}). However, recent work has also focused on
the specific problem of \emph{continuum-armed bandits}, where
$X=[0,1]^p$, and we make some smoothness assumption on the reward
$\mu$; we discuss this work in more detail below.

Many solutions to this problem involve placing a tree structure over
$[0,1]^p$, for example as in Figure~\ref{fig:cab-tree}. The problem can
thus also be thought of as lying within the more general field of \emph
{tree-armed bandits}, where the optimisation occurs over any set with
a known tree structure. Such problems are of interest not only in
noisy optimisation, but also in areas such as artificial intelligence
and online services (see references in \cite{slivkins_multi-armed_2011,yu_unimodal_2011,gelly_grand_2012}).

%
\begin{figure}
\begin{center}
\begin{tikzpicture}[level 1/.style={sibling distance=6cm},
level 2/.style={sibling distance=3cm},
level 3/.style={sibling distance=1.5cm}]
\node{$[0, 1]$}
child {node {$[0, \frac{1}2)$}
child {node {$[0, \frac{1}4)$}
child {node {$\dots$}}
child {node {$\dots$}}}
child {node {$[\frac{1}4, \frac{1}2)$}
child {node {$\dots$}}
child {node {$\dots$}}}}
child {node {$[\frac{1}2, 1]$}
child {node {$[\frac{1}2, \frac{3}4)$}
child {node {$\dots$}}
child {node {$\dots$}}}
child {node {$[\frac{3}4, 1]$}
child {node {$\dots$}}
child {node {$\dots$}}}};
\end{tikzpicture}
\end{center}
\caption{The dyadic tree over $[0,1]$.}
\label{fig:cab-tree}
\end{figure}
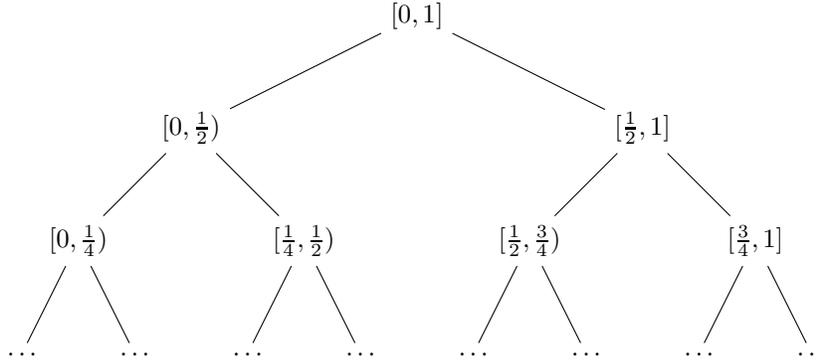

In the following paper, we will describe a new algorithm for noisy
global optimisation, which obtains good cumulative regret under fewer
assumptions than previous results in the literature. As a consequence,
we will also prove new results for continuum-armed and tree-armed
bandits, which may be of wider interest.

We proceed by discussing previous work in more detail, before then
outlining our contributions. The continuum-armed bandit problem was
devised by Agrawal \cite{agrawal_continuum-armed_1995}, and for Lipschitz
reward functions $\mu$, nearly tight bounds on the cumulative regret
were first proved by Kleinberg \cite{kleinberg_nearly_2005}.
Kleinberg applied the UCB1 strategy of
Auer, Cesa-Bianchi and Fischer \cite{auer_finite-time_2002} to a
simple fixed discretisation of the
arm space $[0, 1]$, achieving $\tilde {\mathrm{O}}(T^{2/3})$ regret.

Independently, Cope \cite{cope_regret_2009}
found it was possible to
achieve $\mathrm{O}(\sqrt{T})$ regret given stronger assumptions on $\mu$:
Cope showed this for the stochastic
approximation algorithm of Kiefer and Wolfowitz \cite{kiefer_stochastic_1952}, applied to
unimodal reward functions $\mu$. Auer, Ortner and Szepesv\'ari
\cite{auer_improved_2007} obtained
similar bounds by extending the method of
Kleinberg \cite{kleinberg_nearly_2005}:
Auer, Ortner and Szepesv\'ari obtained $\tilde {\mathrm{O}}(\sqrt{T})$ regret
over any reward function
$\mu$ with, say, finitely many quadratic global maxima.

Kleinberg, Slivkins and Upfal \cite{kleinberg_multi-armed_2008} described a new `zooming'
algorithm, which used an adaptive discretisation of the arm space
$X$, and could be applied whenever $X$ was a metric space. For
Lipschitz $\mu$, Kleinberg, Slivkins and Upfal obtained
regret like $\tilde {\mathrm{O}}(T^{1-1/(\beta+ 2)})$, for a parameter $\beta
\ge0$ they called the \emph{zooming dimension}, measuring the
difficulty of the bandit problem.

Bubeck \textit{et al.} \cite{bubeck_x-armed_2011} described
a related algorithm, HOO, which
could be applied whenever $X$ had a known tree
structure. Bubeck \textit{et al.} again obtained $\tilde
{\mathrm{O}}(\sqrt{T})$ regret over $\mu$ with, say, finitely many quadratic
global maxima, while also covering more general arm spaces and reward
functions.

While the above results are significant, a shared weakness is that
they all\vspace*{1.5pt} require some assumptions on the shape of the reward function
$\mu$. The strongest results, providing $\tilde {\mathrm{O}}(\sqrt{T})$
regret, require us to assume say that $\mu$ has quadratic global
maxima, as in the function
%
\begin{equation}
\label{eq:mu-quadratic} \mu(x) = -x^2.
\end{equation}
However, if we make such an assumption, and then try to optimise a
function with maxima of a different power, such as
%
\begin{equation}
\label{eq:mu-quartic} \mu(x) = -x^4,
\end{equation}
or of mixed powers, such as
%
\begin{equation}
\label{eq:mu-mixed} \mu(x, y) = -x^2 - y^4,
\end{equation}
we will achieve worse rates of regret.

Several authors have tried to improve upon this, constructing bandit
algorithms which adapt to the shape of the reward function. Under
further regularity assumptions, Bubeck, Stoltz and Yu
\cite{bubeck_lipschitz_2011} extended
the algorithm of Kleinberg \cite{kleinberg_nearly_2005} to adapt to the
Lipschitz constant. In a noiseless problem, for the simple regret,
Munos \cite{munos_optimistic_2011}
described an algorithm based on HOO,
which adapts to a wide range of reward functions $\mu$.

In this paper, we will build upon an approach described by
Slivkins \cite{slivkins_multi-armed_2011} for tree-armed
bandits. Slivkins described an
algorithm, TaxonomyZoom, which can adapt to a wide range of reward
functions $\mu$, if the arm space $X$ is given by a finite tree.

Our first contribution is to extend the TaxonomyZoom algorithm to
apply in noisy global optimisation and continuum-armed bandits. We
modify the algorithm to apply directly to infinite arm spaces such as
$[0,1]^p$ (rather than using a fixed discretisation, which could harm
convergence). We also give an explicit estimated maximum $\hat x_T$
(noting that while we could derive a naive choice as in
\cite{bubeck_pure_2009}, ours will be more reliable in practice),
and fix a gap in the proofs of Slivkins.

Our second, more significant contribution is to give a
construction of TaxonomyZoom which can adaptively vary the tree it
optimises over. In previous work on bandits, optimisation has
proceeded either over a fixed partition of the space $X$, or over
partitions selected from a fixed tree. However, in order to get good
convergence rates over functions such as \eqref{eq:mu-mixed}, we will
need to use trees which adapt to the data.

When $X = [0,1]^p$, our algorithm constructs a tree by adaptively
partitioning subsets of $X$ along the axes; we will show that this
procedure achieves optimal convergence rates for a wide variety of
reward functions $\mu$. While the motivation for our algorithm comes
from continuum-armed bandits, our results will apply more generally in
the tree-armed setting, where the tree can be constructed adaptively
from any suitable collection of sub-trees.

Our third contribution is a lower bound on the convergence rate in
tree-armed bandits, given in terms of the zooming dimension $\beta$.
While this result forms part of our lower bound in the continuum-armed
setting, such results have also been missing from previous literature
on tree-armed bandits, and may thus be of wider interest.

Our final contribution is in the interpretation of our results in
noisy global optimisation and continuum-armed bandits. To apply our
algorithm in these settings, we will need to assume the reward
function $\mu$ is sufficiently well-behaved; essentially, that it is
continuous with finitely many polynomial maxima.

The precise condition we will require is that $\mu$ be what we call
\emph{zooming continuous}. This new condition generalises assumptions
previously made for example in Auer, Ortner and Szepesv\'ari \cite{auer_improved_2007} or
Bubeck \textit{et al.} \cite{bubeck_x-armed_2011}, and
gives a concise description of the
reward functions $\mu$ over which we can achieve good cumulative
regret.

When the reward function $\mu$ is zooming continuous, we will show
that our algorithm obtains $\tilde {\mathrm{O}}(\sqrt{T})$ cumulative regret,
and $\tilde {\mathrm{O}}(1/\sqrt{T})$ simple regret, with computation time
$\tilde {\mathrm{O}}(T)$. While the constants in these rates will depend on
$\mu$, our algorithm will attain said rates without prior knowledge
of the rewards.

We note that concurrently with this work,
Valko, Carpentier and Munos \cite{valko_stochastic_2013} have
described another adaptive
algorithm which can be applied to continuum-armed bandits, based on
the approach of Munos \cite{munos_optimistic_2011}. While their results
bound only the simple regret, and do not adapt to reward functions
like \eqref{eq:mu-mixed}, their approach may be easier to generalise,
and their results are complementary to ours.

In Section~\ref{sec:cab}, we will discuss the continuum-armed bandit
problem, and describe the class of zooming continuous reward
functions. In Section~\ref{sec:tab}, we will then describe our algorithm
for tree-armed bandits, and state our results. Finally, in the
supplemental article \cite{bull_supplement_2014} we will give proofs.

\section{Continuum-armed bandits}
\label{sec:cab}

In this section, we describe our results on continuum-armed bandits;
we begin with a precise definition of the multi-armed bandit
problem. Suppose we have a measurable arm space $(X, \mathcal E)$,
and for each $x \in X$, some unknown distribution $P(x)$ over
$[0, 1]$, with mean $\mu(x)$. At time $t$, we are allowed to
choose an arm $x_t \in X$, and then receive a reward $Y_t$ with
distribution $P(x_t)$.

Formally, we take a probability space $(\Omega, \mathcal F, \P)$,
with random variables $Y_t \in[0,1]$ and $Z_t \in Z$, $t \in
\N$, for a measurable space $(Z, \mathcal E')$; the variables
$Z_t$ represent a source of randomisation. At time $t$, we require
that $Z_t$ is distributed independently of past events, $x_t$ is
an $\mathcal E$-measurable function of $Y_1, \dots, Y_{t-1}, Z_1,
\dots, Z_t$, and $Y_t$ has distribution $P(x_t)$, conditionally
on past events and $Z_t$. A \emph{strategy} for the multi-armed
bandit problem is given by the functions $x_t$, and the
distributions of the variables $Z_t$.

If our goal is to optimise $\mu$, we can additionally return an
estimated maximum $\hat x_T$, which we require to be an $\mathcal
E$-measurable function of $Y_1, \dots, Y_T$, $Z_1, \dots, Z_T$,
and an independent randomisation variable $\hat Z_T \in Z$. Our
strategy then also includes the function $\hat x_T$, and the
distribution of the variable $\hat Z_T$.

Finally, we define the cumulative regret $R_T$ as in
\eqref{eq:r-def}, and simple regret $S_T$ as in \eqref{eq:s-def}.
In the following, we will first consider the arm space $X =
[0,1]^p$; our goal will then be to find a strategy which makes the
regrets $R_T$ and $S_T$ as small as possible, for a wide variety
of reward functions~$\mu$.

The functions $\mu$ we consider will satisfy a new condition we call
\emph{zooming continuity}. Essentially, we will require that $\mu$
remains smooth as we `zoom in' on its maxima; Figure~\ref{fig:zcdef}
illustrates the concept.

%
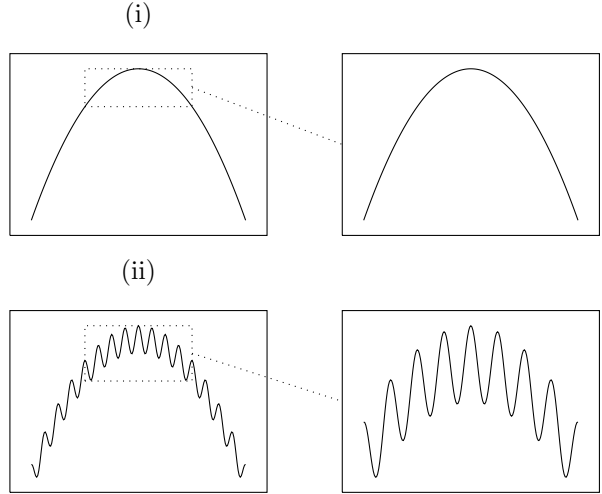
\begin{figure}
\begin{center}
\begin{tikzpicture}
\begin{groupplot}[group style={group size=2 by 2},
height=4cm,width=5cm,ticks=none,samples=201]
\nextgroupplot[title={(i)}]
\addplot[domain=-1:1] {-x^2};
\draw[dotted] (axis cs:-0.5,0) rectangle (axis cs:0.5,-0.25);
\coordinate(l1) at (axis cs:0.5,-0.125);
\nextgroupplot
\addplot[domain=-0.5:0.5] {-x^2};
\nextgroupplot[title={(ii)}]
\addplot[domain=-1:1]
{-x^2+0.1*cos(50*deg(x))};
\draw[dotted] (axis cs:-0.5,0.1) rectangle (axis cs:0.5,-0.3);
\coordinate(l2) at (axis cs:0.5,-0.1);
\nextgroupplot
\addplot[domain=-0.5:0.5]
{-x^2+0.1*cos(50*deg(x))};
\end{groupplot}
\draw[dotted] (l1) -- (group c2r1.west);
\draw[dotted] (l2) -- (group c2r2.west);
\end{tikzpicture}
\end{center}
\caption{Function (i) remains smooth as we zoom in on its maxima; (ii)
does not.}
\label{fig:zcdef}
\end{figure}

As this zooming operation is a common part of algorithms for
continuum-armed bandits, it is natural to require that when doing so,
$\mu$ remains smooth. As such behaviour is neither necessary nor
sufficient for membership in standard smoothness classes, we will thus
require a new definition.

For any set $U \subseteq\R^p$, define its diameter along axis
$i$,
\[
\diam_i(U) = \sup \bigl\{\llvert x_i-y_i
\rrvert : x,y \in U \bigr\},
\]
and its
overall diameter,
\[
\diam(U) = \sup \bigl\{\llVert x-y\rrVert : x,y \in U \bigr\}.
\]
Given also $x \in\R^d$, define its size, relative to $U$, to be
\[
\llVert x\rrVert _U^2 = \sum
_{i=1}^p \biggl(\frac{\llvert  x_i\rrvert }{\diam_i(U)}
\biggr)^2.
\]
We then have the following definition.

\begin{definition}
\label{def:zooming}
Let $X \subset\R^p$ be a compact product of intervals. The
function $f \dvtx  X \to\R$ is \emph{zooming continuous} if:
\begin{enumerate}[(ii)]
\item\label{def:zooming:maxima} $f$ is continuous, with finitely
many global maxima; and
\item\label{def:zooming:local} for any global maximum $x^*$ of
$f$, and any neighbourhood $U$ in $X$ of $x^*$,
%
\begin{equation}
\label{eq:zooming} \sup_{x^*,
U : \diam(U) \le\varepsilon}\frac{\sup_{x, y \in U :
\llVert  x-y\rrVert _U \le\varepsilon}\llvert  f(x) - f(y)\rrvert }{\sup_{z \in U}
\llvert  f(x^*) - f(z)\rrvert } \to0
\end{equation}
as $\varepsilon\to0$.
\end{enumerate}
\end{definition}

We thus require that for any small neighbourhood $U$ of a global
maximum $x^*$, and any points $x,y \in U$ which are close relative
to the size of $U$, the function $f$ does not vary much between
$x$ and $y$, relative to its range over $U$. In other words,
after `zooming in' to $f$ on the set $U$, $f$ remains smooth.

We can show that many functions $\mu$ of interest are zooming
continuous. Essentially, our definition includes any continuous
function $\mu$ with finitely many maxima, each of which behaves like
a suitable polynomial.

\begin{proposition}
\label{lem:zcnd}
Let $X \subset\R^p$ be a compact product of intervals, and $f\dvtx X
\to\R$ be continuous, with finitely many global maxima $x_1^*,
\dots, x_L^*$. For each maximum $x_l^*$, let $f$ satisfy one of
the following as $x \to x_l^*$.
\begin{enumerate}[(ii)]
\item$x_l^*$ is an elliptical maximum,
\[
f(x) = f \bigl(x_l^* \bigr) - \bigl\llVert A_l
\bigl(x-x_l^* \bigr) \bigr\rrVert ^{\alpha_l} \bigl(1 + \mathrm{o}(1)
\bigr),
\]
for
a positive-definite matrix $A_l \in\R^{p\times p}$, and $\alpha_l > 0$.
\item$x_l^*$ is a separable maximum,
\[
f(x) = f \bigl(x_l^* \bigr) - \Biggl(\sum
_{i=1}^p c_{l,i} \bigl\llvert
x_i-x^*_{l,i} \bigr\rrvert ^{\alpha_{l,i}} \Biggr) \bigl(1
+ \mathrm{o}(1) \bigr),
\]
for
$c_{l,i}, \alpha_{l,i} > 0$.
\end{enumerate}
Then $f$ is zooming continuous.
\end{proposition}

The case of elliptical maxima includes all maxima where the function
is locally a quadratic, since we may set $\alpha_l = 2$, and let
$A_l$ be the square root of the Hessian matrix. Alternatively, the
case of separable maxima allows us to model functions which depend
more strongly on some coordinates $x_i$ than others.

We can thus check that zooming continuity includes functions with
maxima like \eqref{eq:mu-quadratic}--\eqref{eq:mu-mixed}, as well as
other combinations of powers. While the conditions of
Lemma~\ref{lem:zcnd} thereby cover our motivating examples, in the
following we will prefer to work directly with the more general and
concise Definition~\ref{def:zooming}.

When the reward function has such behaviour, the following result
shows we can achieve good convergence rates for both the simple and
cumulative regret. This result comes as a corollary to theorems in
Section~\ref{sec:tab}, where we describe a strategy for tree-armed bandits
achieving such rates, and also provide a near-matching lower bound.

\begin{corollary}
\label{cor:zoom-rate}
Let $\varepsilon\in(0, 1)$. There exists a strategy for
continuum-armed bandits, depending only on $\varepsilon$, which
achieves regret
\[
R_T = \tilde{\mathrm{O}}(\sqrt{T}),\qquad S_T = \tilde {\mathrm{O}}(1/
\sqrt{T}),
\]
on
an event with probability $1 - \varepsilon$, whenever the reward
function $\mu$ is zooming continuous. Furthermore, on this event,
the strategy has a total computation time of $\tilde {\mathrm{O}}(T)$.
\end{corollary}

\section{Tree-armed bandits}
\label{sec:tab}

In this section, we will describe our results on the tree-armed bandit
problem. In Section~\ref{sec:tps}, we will give a definition of the
problem, and in Section~\ref{sec:alg}, describe the algorithm we will use
to solve it. In Section~\ref{sec:wbr}, we will define a class of reward
functions over which our algorithm performs well, and in
Section~\ref{sec:results}, state our bounds on its regret and complexity.

\subsection{Problem statement}
\label{sec:tps}

In the tree-armed bandit problem, we again consider the multi-armed
bandit problem described in Section~\ref{sec:cab}, but now with a more
general arm space $X$. We allow any space $X$ on which we are
given a certain tree structure, which we define below; we will show
that the continuum-armed bandit problem is a special-case of this more
general setting.

To define our setting, let the arm space $X$ be a Cartesian product
$\prod_{i=1}^p X_i$, for coordinate spaces $X_i$. For $i = 1,
\dots, p$, let $\mathcal T_i$ be a tree with root node $X_i$, and
whose nodes are all given by non-empty subsets of $X_i$. Further
require that each node $U$ is either a leaf node, or has children
$V$ which form a partition of the set $U$. Each non-leaf node must
have at least 2 and at most $q$ children, for a constant $q \in
\N$.

Formally, we will also require a $\sigma$-algebra $\mathcal E$ on
$X$, defined in terms of the trees $\mathcal T_i$. For each
coordinate space $X_i$, let $\mathcal E_i$ be the sigma-algebra
generated by the nodes $U$ of $\mathcal T_i$. We then define
$\mathcal E$ to be the product $\sigma$-algebra of the $\mathcal
E_i$.

As before, we sequentially choose arms $x_t \in X$, and receive
rewards $Y_t\in[0,1]$; our goal remains to find a strategy
minimising the regrets $R_T$ and $S_T$, for a wide variety of
reward functions~$\mu$. However, we must now do so for general treed
spaces $X$, given only the trees~$\mathcal T_i$.

Continuum-armed bandits lie within this setting, letting each
coordinate space $X_i = [0,1]$. The trees $\mathcal T_i$ can be
chosen to be \emph{dyadic trees} on $[0, 1]$, defined as follows. The
dyadic tree on $[0, 1)$ is the tree with root node $[0, 1)$, and
where each node $[a, b)$ has children $[a, \frac12(a+b))$,
$[\frac12(a+b), b)$.

We can similarly define the dyadic tree on $[0, 1]$, instead
allowing each node with upper bound 1 to contain the point 1; this
tree is illustrated in Figure~\ref{fig:cab-tree}. If the trees
$\mathcal
T_i$ are dyadic trees on $[0, 1]$, then $\mathcal E$ is the Borel
$\sigma$-algebra on $[0,1]^p$, and we recover the setting of
Section~\ref{sec:cab}.

With these definitions, we can now consider continuum-armed bandits as
a special-case of tree-armed bandits. In the following section, we
will describe an algorithm for solving tree-armed bandits, which when
applied to continuum-armed bandits, achieves the bounds in
Corollary~\ref{cor:zoom-rate}.\vspace*{-2pt}

\subsection{Adaptive-treed bandits}\vspace*{-1pt}
\label{sec:alg}

Our algorithm proceeds in much the same fashion as the TaxonomyZoom
algorithm of Slivkins \cite{slivkins_multi-armed_2011}. At time $t$, we
partition the arm space $X$ into a set $\mathcal A_{t-1}$ of \emph
{active boxes}, chosen in terms of the past rewards $Y_1, \dots,
Y_{t-1}$. For each box $B \in\mathcal A_{t-1}$, we compute an \emph
{index} $I_{t-1}(B) \in\R$, which upper bounds its maximum reward
$\sup_{x \in B} \mu(x)$. We then select an active box $B_t$
maximising the index $I_{t-1}$, and pull an arm $x_t$ chosen
uniformly at random from $B_t$.

To describe the algorithm in detail, we will need some additional
definitions. We begin with the concepts which depend on the sample
space $X$: the set of boxes $B \subseteq X$ we will use to
construct our partitions, and the distribution $\pi$ over $X$ we
will treat as uniform.

In the specific case of continuum-armed bandits, the boxes $B$ will
be the products of dyadic intervals in $[0,1]^p$, and $\pi$ will
be the uniform distribution on $[0,1]^p$. However, since our methods
also apply to the more general tree-armed setting, we now give more
general descriptions of these ideas.

We define a box $B$ to be any product $\prod_{i=1}^p U_i$, where
each $U_i$ is a node in the tree $\mathcal T_i$; we further let
$\mathcal B$ denote the set of all such boxes. For a fixed reward
function $\mu\dvtx  X \to[0,1]$, we also define the width $W$ of a
box $B$ to be
\[
W(B) = \sup_{x \in B} \mu(x) - \inf_{x \in B}
\mu(x).
\]

We next define a distribution $\pi$ on the measurable space $(X,
\mathcal E)$, given as the product of distributions $\pi_i$ on the
spaces $(X_i, \mathcal E_i)$. Intuitively, $\pi_i$ will be the
distribution of a point in $X_i$ chosen by uniform random descent of
the tree $\mathcal T_i$.

To be precise, we generate a random sequence of nodes $U_n$ in
$\mathcal T_i$, setting $U_1 = X_i$. For $n \in\N$, if $U_n$
is a leaf node, we terminate the sequence at $U_n$; otherwise, we
choose $U_{n+1}$ uniformly at random from the children of $U_n$.
We then define a distribution $\pi_i$ on $(X_i, \mathcal E_i)$~by\vspace*{-2pt}
%
\begin{equation}
\label{eq:pi-def} \pi_i(U) = \P(\exists n \in\N: U_n =
U), \qquad U \in\mathcal T_i.
\end{equation}
It can be checked this uniquely defines a distribution $\pi_i$ on
$(X_i, \mathcal E_i)$.

We have thus defined the set $\mathcal B$ of boxes we will use to
partition $X$, and the distribution $\pi$ over $X$ we will take
as uniform. We note that for continuum-armed bandits, these
definitions agree with those given above.

In the following, we will also wish to sample from $\pi$; in the
case of continuum-armed bandits this is easy, as $\pi$ is the
uniform distribution. More generally, we will assume that $\pi$ can
be easily sampled from; note that we can always generate an
approximate sample by random descent of the trees $\mathcal T_i$.
Typically we will expect the $\sigma$-algebra $\mathcal E$ to be
fine enough to define this sample to our satisfaction, but if not, we
allow any sample satisfying \eqref{eq:pi-def}.\vadjust{\goodbreak}

We now move onto the definition of the index $I_t$. For each active
box $B$, $I_t(B)$ will be based on the empirical mean of past
rewards $Y_s$ associated with arms $x_s$ in $B$. To ensure this
is an upper bound for the maximum reward over $B$, we will add two
additional terms: one to correct for the stochastic error associated
with estimating the mean reward, and one to bound the difference
between mean and maximum.

Suppose that at time $t$, we select the active box $B_t$, drawing
$x_t$ from the distribution $\pi\mid B_t$. For any box $B$, we
will say $B$ was \emph{hit} at time $t$ if $x_t \in B \subseteq
B_t$. Let $n_t(B)$ be the number of times $s \le t$ at which
$B$ was hit, and if $n_t(B) > 0$, let $\mu_t(B)$ be the
corresponding average reward. For fixed $\mu$, we note that
$\mu_t(B)$ is an unbiased estimate of
\[
\mu(B) = \E_{\pi} \bigl[\mu(x) \mid x \in B \bigr],
\]
the expected reward on $B$ under $\pi$.

To bound the error in this estimate, we next define a confidence
radius $r_t(B)$, chosen so that $\llvert  \mu_t(B) - \mu(B)\rrvert
\le
r_t(B)$ with high probability. We first fix an error probability
$\varepsilon\in(0,1)$, which will control the accuracy of our
bound; we will show that our results on the regret hold with
probability $1 - \varepsilon$.

For any box $B = \prod_{i=1}^p U_i$, we then let $d(B)$ denote the
\emph{depth} of $B$, the maximum depth of any $U_i$ in its
corresponding tree $\mathcal T_i$, and define the constant
\[
\rho(B) = q^{p(d(B)+1)}.
\]
We also set $\tau= 4\varepsilon^{-1}$, and then define the
confidence radius
%
\begin{equation}
\label{eq:cr} r_t(B) = 2\sqrt{\log \bigl[\rho(B) \bigl(\tau+
n_t(B) \bigr) \bigr]/n_t(B)}.
\end{equation}

To conclude the definition of the index $I_t(B)$, we will need a
term bounding the difference between the mean and maximum reward on
$B$. This term will depend on a constant $\gamma\in(0, 1)$
called the \emph{quality}, a concept we inherit from
Slivkins \cite{slivkins_multi-armed_2011}.

The quality $\gamma$ describes how difficult we expect the
tree-armed bandit problem to be, and thus how conservatively our
algorithm should act. In the following sections, we discuss the
implications of $\gamma$ in more detail. For now, we note that
smaller $\gamma$ corresponds to a more difficult problem, and more
conservative behaviour.

Given a fixed choice of $\gamma$, we then define the index
\[
I_t(B) = \mu_t(B) + (1 + 2p\nu)r_t(B),
\]
where the constant
\[
\nu= 8\sqrt{2\gamma^{-1}\log_2 \bigl(2
\gamma^{-1} \bigr)};
\]
if $n_t(B) = 0$, we take $I_t(B) = +\infty$.
The index $I_t(B)$ is thus a sum of the empirical mean $\mu_t(B)$,
the confidence radius $r_t(B)$, and an additional term $2p\nu
r_t(B)$, which bounds the difference between the mean and maximum
reward over $B$.

%
\begin{figure}[b]
\begin{center}
\begin{tikzpicture}
\begin{groupplot}[group style={group size=2 by
1},height=6cm,width=6cm,ticks=none,
enlargelimits=false,xmin=0,xmax=1,ymin=0,ymax=1,
unbounded coords=jump,no markers]
\nextgroupplot[title={(i)}]
\addplot[black] coordinates {
(0.5, 0) (0.5, 1) (inf, inf)
(0, 0.5) (0.5, 0.5)
};
\node at (axis cs:0.25, 0.75) {$B$};
\nextgroupplot[title={(ii)}]
\addplot[black] coordinates {
(0.5, 0) (0.5, 1) (inf, inf)
(0, 0.5) (0.5, 0.5) (inf, inf)
(0.25, 0.5) (0.25, 1)
};
\node at (axis cs:0.125, 0.75) {$C_1$};
\node at (axis cs:0.375, 0.75) {$C_2$};
\end{groupplot}
\end{tikzpicture}
\end{center}
\caption{Plot (i) shows a partition $\mathcal A_t$ of the arm space
$X = [0,1]^2$ into active boxes; (ii) shows $\mathcal A_t$ after
the box $B$ has been split to maintain Invariant~\protect\ref
{inv:bias}. The
boxes $C_1, C_2$ satisfy condition \protect\eqref{eq:boxes-agree}.}
\label{fig:abex}
\end{figure}
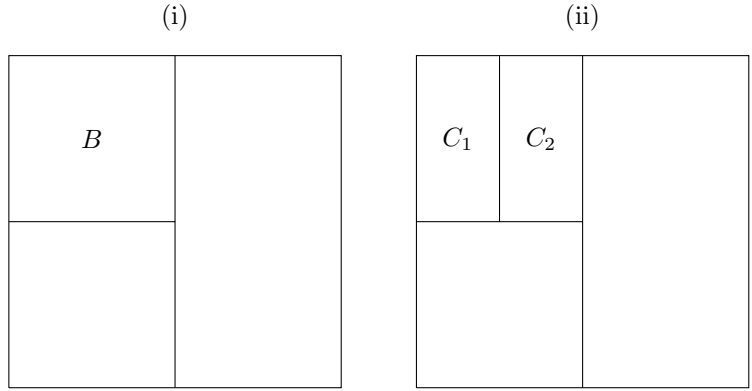

We next describe our set $\mathcal A_t$ of active boxes. Our goal
will be to choose as few active boxes as possible, while still
ensuring that for each active box $B$, the index $I_t(B)$ is an
upper bound for the maximum reward over $B$. To do so, we will aim
to select a set\vadjust{\goodbreak} of active boxes $B$ satisfying the inequality $W(B)
\le2p\nu r_t(B)$; we will thus need to find estimates of the widths
$W(B)$.

The estimates will work on the principle that, if the reward function
$\mu$ is well-behaved, we will be able to find large enough
sub-boxes $C_1, C_2 \subseteq B$ for which $\mu(C_1) - \mu(C_2)
\approx W(B)$. Since we always have $\mu(C_1) - \mu(C_2) \le W(B)$,
we may thus estimate $W(B)$ by a suitable maximum of these
differences, taken over many pairs $C_1, C_2$.

Since we will not have access to the means $\mu(C_k)$ themselves, we
will need to bound them using the data. We therefore define the lower
and upper bounds on the mean reward,
\[
\underline\mu_t(B) = \mu_t(B) - r_t(B),
\qquad\overline\mu_t(B) = \mu_t(B) +
r_t(B).
\]
We may then define our width estimate
\[
W_t(B) = \max_{(C_1,C_2) \in\mathcal M(B)} \underline
\mu_t(C_1) - \overline\mu_t(C_2).
\]

The maximum is taken over the set $\mathcal M(B)$ of all pairs
$(C_1, C_2)$ of boxes $C_1, C_2 \subseteq B$, which for $k = 1,
2$ satisfy:
\begin{enumerate}[(ii)]
\item$\pi(C_k \mid B) \ge\gamma$; and
\item for some $i = 1, \dots, p$, we have $U_{i,k} \in\mathcal
T_i$, and $U_j \in\mathcal T_j, j\ne i$, satisfying
%
\begin{equation}
\label{eq:boxes-agree} C_k = U_1 \times\cdots\times
U_{i-1} \times U_{i,k} \times U_{i+1} \times\cdots
\times U_p.
\end{equation}
\end{enumerate}
In other words, $\mathcal M(B)$ contains all pairs $(C_1, C_2)$
of boxes in $B$ which are not much smaller than $B$, and agree
except along one axis; one such pair is illustrated in
Figure~\ref{fig:abex}.

\begin{algorithm}[b]
\caption{Adaptive-treed bandits (ATB)}
\label{alg:atb}
\KwData{space $X$, trees $\mathcal T_i$, error rate
$\varepsilon$, quality $\gamma$}
set $\mathcal A_0 = \{X\}$\;
\For{$t=1, \dots, T$}{
select a box $B_t \in\mathcal A_{t-1}$ maximising $I_{t-1}$\;
play an arm $x_t$ drawn at random from $\pi\mid B_t$\;
set $\mathcal A_t = \mathcal A_{t-1}$\;
\While{Invariant~\textup{\ref{inv:bias}} is violated, by $B =
\prod_{j=1}^p U_j$, and $C_1, C_2$ differing along axis
$i$}{
remove $B$ from $\mathcal A_t$\;
\For{$V$ a child of $U_i$ in $\mathcal T_i$}{
add $U_1 \times\cdots\times U_{i-1} \times V \times
U_{i+1} \times\cdots\times U_p$ to $\mathcal A_t$\;}}}
\Return{$x_{T^*}$}
\end{algorithm}

Having defined our width estimates $W_t(B)$, we now return to the
set $\mathcal A_t$ of active boxes. We first state that at the
beginning of the algorithm, only the root box $X$ is active:
$\mathcal A_0 =  \{ X  \}$. At later times $t$, we define
each $\mathcal A_t$ in terms of $\mathcal A_{t-1}$, so as to
maintain the following invariant.

\begin{invariant}
\label{inv:bias}
Either:
\begin{enumerate}[(ii)]
\item$n_t(B) = 0$ for some $B \in\mathcal A_t$; or
\item$W_t(B) < \nu r_t(B)$ for all $B \in\mathcal A_t$.
\end{enumerate}
\end{invariant}

We start by setting with $\mathcal A_t = \mathcal A_{t-1}$. Suppose
this violates Invariant~\ref{inv:bias}, so we have $W_t(B) \ge\nu r_t(B)$
for some box $B = \prod_{j=1}^p U_j$;\vspace*{2pt} then let $W_t(B)$ be
maximised by boxes $C_1, C_2$ differing only along axis $i$. We
remove $B$ from $\mathcal A_t$, and replace it with the boxes
$U_1 \times U_{i-1} \times V \times U_{i+1} \times U_p$, for all
children $V$ of $U_i$ in $\mathcal T_i$.

We repeat this process until $\mathcal A_t$ satisfies
Invariant~\ref{inv:bias}; we note the process must terminate, as each step
increases the number of active boxes $B$, without creating
additional design points $x_s$. The process is illustrated in
Figure~\ref{fig:abex}.

We have thus described how we choose the set $\mathcal A_t$ of
active boxes. Finally, we define our estimate $\hat x_T$ of a global
maximum of $\mu$; we set $\hat x_T$ = $x_{T^*}$, where the
optimal time
\[
T^* = \argmin_{t=1}^T r_t(B_t),
\]
breaking ties arbitrarily.

We have then described in full our algorithm ATB, given by
Algorithm~\ref{alg:atb}. We note that our algorithm is closely related to
the TaxonomyZoom algorithm of Slivkins
\cite{slivkins_multi-armed_2011}; we
briefly describe the changes.

First, to allow us to work with infinite trees, we have altered the
confidence radius $r_t(B)$ and constant $\nu$. Second, to give
an explicit algorithm for noisy global optimisation, we have included
a rule for choosing an optimal point $\hat x_T$. Third, we have
altered Invariant~\ref{inv:bias} to allow an easier bound on the
computational complexity.

Last, we have made a number of changes which allow us to work with
multiple trees $\mathcal T_i$. The first of these is that we
partition the arm space $X$ into boxes $B \in\mathcal B$ given by
a product of nodes in trees, rather than the nodes themselves. The
second is that we have altered the width estimate $W_t(B)$ to
require that the boxes $C_1, C_2$ agree except in one axis; this
allows us to detect not only the width of a box $B$, but also an
axis $i$ along which it varies.

The final change is in the procedure for ensuring that
Invariant~\ref{inv:bias} holds. When the invariant is violated by a box
$B$, we split that box only along the axis $i$; this process
allows us to adapt the shape of the active boxes $B$ to the shape of
the reward function~$\mu$.

\subsection{Well-behaved rewards}
\label{sec:wbr}

We now describe the conditions we will require on the reward function
$\mu$. Our conditions will be motivated by Definition~\ref{def:zooming},
and we will see that they hold in continuum-armed bandits whenever
$\mu$ is zooming continuous. We will state the conditions more
generally for the tree-armed case, however, as this allows us to both
argue more directly, and also compare our conditions with those in
previous work.

To begin, we will need some preliminary definitions. In the following,
we will consider collections $\mathcal C$ of disjoint boxes $B \in
\mathcal B$. We will say a box $B$ is \emph{on} $\mathcal C$, if
it is a union of boxes in $\mathcal C$. We will further say
$\mathcal C$ is a \emph{refinement} of $\mathcal C'$, if this is
true for all $B \in\mathcal C'$.

A specific type of collection $\mathcal C$ we will consider is the
\emph{grid}. A grid $\mathcal G$ is any set of boxes
$ \{\prod_{i=1}^p U_i : U_i \in\mathcal S_i \}$, where for
each $i = 1, \dots, p$, $\mathcal S_i$ is a collection of disjoint
nodes in $\mathcal T_i$. We will say grids $\mathcal G_1, \dots,
\mathcal G_L$ are \emph{separated}, if for any box $B$ on
$\bigcup_{l=1}^L \mathcal G_l$, $B$ is on a single $\mathcal
G_l$.

Finally, for a fixed reward function $\mu$, given $\delta> 0$ we
define the level set
\[
\mathcal X_\delta= \bigl\{x \in X : \mu^* - \mu(x) \le\delta \bigr\},
\]
and for any
box $B$, we define its maximum and average badness,
%
\begin{equation}
\label{eq:deltas} \delta(B) = \mu^* - \min_{x \in B} \mu(x), \qquad
\Delta(B) = \mu^* - \mu(B).
\end{equation}
We are now ready to state our conditions on $\mu$.

\begin{definition}
\label{def:wbr}
Let $\mu\dvtx X \to[0, 1]$ be $\mathcal E$-measurable. We will
say $\mu$ is \emph{well-behaved} if for each $m \in\N$, we have
a partition $\mathcal B_m$ of $X$, made up of boxes $B \in
\mathcal B$, and a subset $\mathcal C_m \subseteq\mathcal B_m$,
satisfying the following.

\begin{enumerate}[(iii)]
\item\label{it:wbr-cover} For each $m \in\N$, letting $\delta_m
= 2^{1-m}$, the level set $\mathcal X_{\delta_m}$ is covered by
$\mathcal C_m$.
\item\label{it:wbr-card} Each $\mathcal C_m$ has cardinality at most
$\kappa
\delta_m^{-\beta}$, for constants $\kappa> 0$, $\beta\ge
0$.
\item\label{it:wbr-boxes} For each $m \in\N$, the boxes $B \in
\mathcal C_m$ satisfy:
\begin{enumerate}[(b)]
\item[(a)]\label{it:wbr-width}$W(B) \le\delta_m/12p$, and
\item[(b)]\label{it:wbr-depth}$d(B) \le\lambda m$, for a constant
$\lambda> 0$.
\end{enumerate}
\item\label{it:wbr-quality} For each box $B$ on some $\mathcal
C_m$, there exist two sub-boxes $C_1, C_2 \subseteq B$
satisfying condition \eqref{eq:boxes-agree}, with:
\begin{enumerate}[(b)]
\item[(a)]$\pi(C_k \mid B) \ge\gamma$, $k =
1, 2$, for a constant $\gamma\in(0, 1)$, and
\item[(b)]$\mu(C_1) - \mu(C_2) \ge\frac1p(W(B)
- \frac14
\delta(B))$.
\end{enumerate}
\item\label{it:wbr-grids} For each $m \in\N$, we have some $L_m
\in\N$, and separated grids $\mathcal G_{1,m}, \dots,
\mathcal G_{L_m,m}$, such that $\mathcal C_m \subseteq
\bigcup_{l=1}^{L_m} \mathcal G_{l,m}$.
\item\label{it:wbr-finer} Each $\mathcal B_{m+1}$ is a refinement
of $\mathcal B_m$.
\end{enumerate}

We will call $\beta$ the \emph{zooming dimension}, and $\gamma$
the \emph{quality}.
\end{definition}

We next discuss the implications of our definition, which is
illustrated in Figure~\ref{fig:wbr}. Firstly, we note that the
conditions are all satisfied when the reward function $\mu$ is
zooming continuous.

%
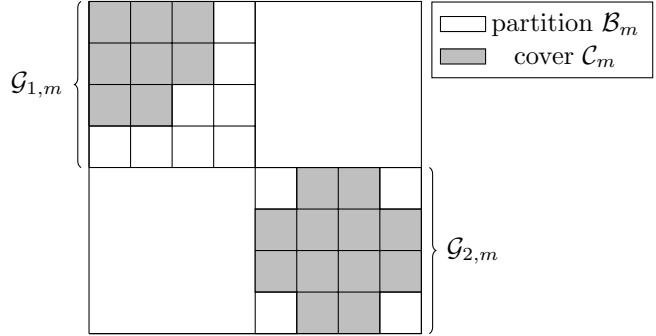
\begin{figure}
\begin{center}
\begin{tikzpicture}
\begin{axis}[height=6cm,width=6cm,ticks=none,clip=false,
enlargelimits=false,unbounded coords=jump,no markers,
legend style={legend pos=outer north east},area legend]
\addplot[black] coordinates {
(0.5, 0) (0.5, 1) (inf, inf)
(0, 0.5) (1, 0.5)
};
\addlegendentry{partition $\mathcal B_m$}
\addplot[black,fill=lightgray] coordinates {
(0, 0.625) (0, 1) (0.375, 1) (0.375, 0.75)
(0.25, 0.75) (0.25, 0.625) (0, 0.625) (inf, inf)
(0.625, 0) (0.875, 0) (0.875, 0.125) (1, 0.125)
(1, 0.375) (0.875, 0.375) (0.875, 0.5) (0.625, 0.5)
(0.625, 0.375) (0.5, 0.375) (0.5, 0.125) (0.625, 0.125)
(0.625, 0)
} \closedcycle;
\addlegendentry{cover $\mathcal C_m$}
\draw[step=12.5] (axis cs:0, 0.5) grid (axis cs:0.5, 1);
\draw[step=12.5] (axis cs:0.5, 0) grid (axis cs:1, 0.5);
\draw[xshift=-3,decorate,decoration={brace}]
(axis cs:0, 0.5) -- (axis cs:0, 1) node[midway,left,xshift=-3]
{$\mathcal G_{1,m}$};
\draw[xshift=3,decorate,decoration={brace}]
(axis cs:1, 0.5) -- (axis cs:1, 0) node[midway,right,xshift=3]
{$\mathcal G_{2,m}$};
\end{axis}
\end{tikzpicture}
\end{center}
\caption{A partition $\mathcal B_m$ of the arm space $X =
[0,1]^2$, together with the cover $\mathcal C_m$, and grids $\mathcal G_{l,m}$.}
\label{fig:wbr}
\end{figure}

\begin{theorem}
\label{thm:zooming}
Let the arm space $X = [0,1]^p$, given as the product of
coordinate spaces $X_i = [0, 1]$, $i = 1, \dots, p$, with dyadic
trees $\mathcal T_i$ over each $X_i$. If $\mu\dvtx X \to[0, 1]$ is
zooming continuous, then $\mu$ is well-behaved, with zooming
dimension $\beta= 0$.
\end{theorem}

Second, we note that the conditions of Definition~\ref{def:wbr} are related
to other conditions previously studied in the literature. The zooming
dimension $\beta\ge0$, and quality $\gamma\in(0,1)$, are
related to similar concepts defined by
Kleinberg, Slivkins and Upfal \cite{kleinberg_multi-armed_2008} and
Slivkins \cite{slivkins_multi-armed_2011}, and measure the difficulty of
solving a bandit problem with reward function $\mu$, when
subdividing the arm space $X$ using the trees $\mathcal T_i$.

We will discuss in more detail the meaning of these quantities below;
for now, we note that they are a function both of the reward function
$\mu$, and the trees $\mathcal T_i$. In the following, we will
assume that we have some natural choice of trees $\mathcal T_i$ we
may treat as fixed, as is the case in continuum-armed bandits; we may
thus consider these quantities primarily as a function of $\mu$.

Intuitively, conditions \eqref{it:wbr-cover}--\eqref{it:wbr-boxes}(\hyperref[it:wbr-width]{a})
state that $\mu$ has zooming dimension $\beta\ge0$. This concept
was introduced by Kleinberg, Slivkins and Upfal \cite{kleinberg_multi-armed_2008}, and bounds the
number of near-maximal boxes $B$ we must evaluate to find the global
maxima of $\mu$. The larger $\beta$ is, the more alternatives we
must consider, and the worse our regret rates will be.

Kleinberg, Slivkins and Upfal \cite{kleinberg_multi-armed_2008} defined the zooming dimension
relative to a fixed metric, with respect to which $\mu$ is assumed
to be Lipschitz. Our formulation is more closely related to that of
Slivkins \cite{slivkins_multi-armed_2011}, who did not fix a metric, but
instead used the strongest metric which $\mu$ is Lipschitz with
respect to.

Our condition improves upon that of Slivkins \cite{slivkins_multi-armed_2011}
by allowing the cover $\mathcal C_m$ to be made up of boxes $B \in
\mathcal B$, constructed not just from a single tree $\mathcal
T_i$, but also from arbitrary combinations of them. This flexibility
allows us to ensure that a wider variety of reward functions $\mu$
will have zooming dimension $\beta= 0$; in particular, it is
necessary to get near-optimal rates for the separable maxima in
Lemma~\ref{lem:zcnd}.

For the continuum-armed bandit problems we will consider, we will
always have zooming dimension $\beta= 0$. However, in tree-armed
bandits, we will also consider the case $\beta> 0$, as this allows
our results to hold in more generality. In particular, we will prove
near-matching lower bounds on the regret in terms of all $\beta\ge
0$.

Condition \eqref{it:wbr-boxes}(\hyperref[it:wbr-depth]{b}) controls the depth of near-maximal
boxes $B$; assuming this condition allows us to construct an
algorithm which is more computationally efficient. A similar approach
is considered by Bubeck \textit{et al.} \cite{bubeck_x-armed_2011}, who discuss artificially
truncating trees at a certain depth.

Intuitively, condition \eqref{it:wbr-quality} states that $\mu$ has
quality $\gamma\in(0, 1)$. This concept was introduced by
Slivkins \cite{slivkins_multi-armed_2011}, and bounds the difficulty in
estimating the widths $W(B)$. Our version of this condition is new,
and improves upon Slivkins' in two ways.

First, we require the bound to hold for a larger collection of boxes
$B$; we will show this change allows us to fix a gap in the argument
of Slivkins \cite{slivkins_multi-armed_2011}. Second, we require the boxes
$C_1, C_2$ to satisfy condition \eqref{eq:boxes-agree}. In the case
$p = 1$, when we have a single tree over $X$, this condition is
trivial. However, when $p > 1$, it allows our algorithm to detect
the axes along which $\mu$ varies, and so adaptively combine the
trees $\mathcal T_i$.

Conditions \eqref{it:wbr-grids} and \eqref{it:wbr-finer} are new to
this work, and are also required to work with multiple trees
efficiently. Again, when $p = 1$ the conditions can be shown to be
trivially satisfiable; when $p > 1$, they will be necessary to prove
our adaptive results.

Condition \eqref{it:wbr-grids} requires that the near-maximal boxes
$B$ lie within a grid structure; that the boxes can be created by
independent subdivisions of the axes $X_i$. This condition will be
necessary to ensure that when our algorithm subdivides the axes, it
does not create too many active boxes.

Condition \eqref{it:wbr-finer} requires that the near-maximal boxes
$B \in\mathcal C_m$ become smaller as $m$ increases; that they
describe consistent regions of the arm space $X$ as $\delta_m \to
0$. This condition will be necessary to ensure that as our algorithm
progresses, the active boxes created at earlier time steps do not
hinder us at later ones.

While the main motivation behind Definition~\ref{def:wbr} is our application
to continuum-armed bandits, our results can also be applied to other
tree-armed bandit problems, including those with finite trees. We note
that while our definitions do not require it, it will be easiest to
satisfy Definition~\ref{def:wbr} when all leaf nodes $U_i$ in trees
$\mathcal T_i$ are singleton sets, a condition which should be
satisfied by any reasonable choice of trees $\mathcal T_i$.

\subsection{Results for tree-armed bandits}
\label{sec:results}

We now give our regret bounds for tree-armed bandits. We will prove
our results uniformly over a class of reward functions $\mu$, which
we describe below.

For an arm space $X$, given as the product of coordinate spaces
$X_i$, $i = 1, \dots, p$, each equipped with tree $\mathcal
T_i$, a zooming dimension $\beta\ge0$, a quality $\gamma\in(0,
1)$, and constants $\kappa, \lambda> 0$, let
\[
\mathcal P = \mathcal P(X, \mathcal T, \beta, \gamma, \kappa, \lambda)
\]
denote the class of arm distributions $P$ whose reward
functions $\mu$ are well-behaved, with the above constants. We note
that the class $\mathcal P$ is increasing in the parameters
$\beta$, $\kappa$ and $\lambda$, and decreasing in $\gamma$.

We also fix some notation we will use to describe our rates of
regret. Given functions $f, g \dvtx  \N\to\R$ satisfying $f(T) =
\mathrm{O}(g(T))$ as $T \to\infty$, we write $f(T) \lesssim g(T)$, and
$g(T) \gtrsim f(T)$. If both $f(T) \lesssim g(T)$ and $f(T)
\gtrsim g(T)$, we write $f(T) \approx g(T)$.

We now begin by establishing a lower bound on the regret any algorithm
can achieve, in our setting of the tree-armed bandit problem. Our
argument works by reducing to a finite arm space, and then applying a
lower bound of Bubeck \cite{bubeck_jeux_2010}.

\begin{theorem}
\label{thm:lower-bound}
Suppose the trees $\mathcal T_i$ have no leaf nodes, and fix $\beta
\ge0$. For large enough $\kappa, \lambda> 0$, small enough
$\gamma, \varepsilon\in(0,1)$, and any strategy for tree-armed
bandits, we have events $E_T$ and $E_T'$, each of probability at
least $\varepsilon$ under some $P \in\mathcal P$, for which
\[
R_T \ge Tr_T \qquad \mbox{on } E_T,\qquad
S_T \ge r_T \qquad \mbox{on } E_T',
\]
for a rate
\[
r_T \gtrsim T^{-1/(\beta+2)}.
\]
\end{theorem}

This rate matches, up to log factors, the rates in upper bounds which
have previously been proved, for example for the zooming algorithm of
Kleinberg, Slivkins and Upfal \cite{kleinberg_multi-armed_2008}, or the HOO algorithm of
Bubeck \textit{et al.} \cite{bubeck_x-armed_2011}. In the
following, we will show that it
also matches upper bounds for the adaptive algorithm described in
this paper.

We begin by showing that, up to log factors, Algorithm~\ref{alg:atb}
achieves the same rates, given only knowledge of the quality
$\gamma$. We note that a similar result was stated by
Slivkins \cite{slivkins_multi-armed_2011}, in the case of a single finite
tree. In the following, we use a novel argument to fix a gap in the
argument of Slivkins,\footnote{The proof
of Slivkins' Lemma~4.4(b) incorrectly
assumes that all deactivated boxes have been selected.} and also
extend the result to multiple, infinite trees $\mathcal T_i$.

\begin{theorem}
\label{thm:regret}
Fix $\varepsilon, \gamma\in(0,1)$. Running Algorithm~\ref{alg:atb} with
error rate $\varepsilon$ and quality $\gamma$, for any $\beta\ge
0$ and $\kappa, \lambda> 0$, we have events $E_T$, of
probability at least $1-\varepsilon$ under any $P \in\mathcal P$,
on which
\[
R_T \le Tr_T, \qquad S_T \le
r_T,
\]
for a rate
\[
r_T \lesssim \biggl(\frac{
T}{\gamma^{-1}\log(\gamma^{-1})\log(\varepsilon^{-1} + T)\log
(T)^{1(\beta= 0)}} \biggr)^{-1/(\beta+2)},
\]
uniformly in $\gamma$ and $\varepsilon$.
\end{theorem}

We have thus shown that Algorithm~\ref{alg:atb} achieves good rates of
regret, without detailed knowledge of the reward function $\mu$.
Furthermore, the algorithm adapts to the shape of $\mu$ not only
within a single tree $\mathcal T_i$, but also by combining the trees
in whichever way minimises the zooming dimension $\beta$.

In the above theorem, Algorithm~\ref{alg:atb} still required a bound
$\gamma$ on the quality of $\mu$. As a corollary, however, we can
achieve similar rates of regret, up to say an additional log factor,
without prior knowledge of $\mu$.

\begin{corollary}
\label{cor:doubling}
Fix $\varepsilon\in(0, 1)$. Running Algorithm~\ref{alg:atb} with error
rate $\varepsilon$ and quality $\log(T)^{-1}$, for any $\beta\ge
0$, $\kappa, \lambda> 0$ and $\gamma\in(0, 1)$, we have
events $E_T$, of probability at least $1-\varepsilon$ under any $P
\in\mathcal P$, on which
\[
R_T \le Tr_T, \qquad S_T \le
r_T,
\]
for a rate
\[
r_T \lesssim \biggl(\frac{T}{\log(\varepsilon^{-1} +
T)\log(T)^{1+1(\beta= 0)}\log(\log(T))} \biggr)^{-1/(\beta+2)},
\]
uniformly in $\varepsilon$.
\end{corollary}

We note that in the above construction, Algorithm~\ref{alg:atb} is no longer
an anytime algorithm, as its quality parameter depends on the time
horizon $T$. If an anytime algorithm is desired, one can be
constructed using the doubling trick, as in
Slivkins \cite{slivkins_multi-armed_2011}; however, we need not consider this
further.

We have thus shown that Algorithm~\ref{alg:atb} can achieve near-optimal
rates of regret, for the optimal combination of trees $\mathcal
T_i$, without prior knowledge of $\mu$. We note that, together with
Theorem~\ref{thm:zooming}, we can use this result to deduce the first part
of Corollary~\ref{cor:zoom-rate}, our result establishing good rates of
regret in continuum-armed bandits.

It remains to discuss the implementation of our algorithm; we will
show that, for a careful implementation, it can run in almost linear
time. The key idea is to store the active boxes $B$ in a priority
queue, with priority given by their index $I_t(B)$. The operation of
choosing a box $B_t$ with maximal index can then be performed in
constant time.

The remaining work lies in efficiently maintaining the set $\mathcal
A_t$ of active boxes, and their indices $I_t$. We note that for
active boxes $B$, the index $I_t(B)$, width estimate $W_t(B)$,
and confidence radius $r_t(B)$ are changed only when we choose an
arm $x_t \in B$. We thus need ensure only that these quantities can
be updated efficiently when given a new data point.

To do so, we will keep some preliminary computations stored in
memory. For each active box $B \in\mathcal A_t$, we store a list of
the past data points $(x_s, Y_s)$, $s \le t$, for which $x_s \in
B$. For each box $C \subseteq B$ satisfying $\pi(C \mid B) \ge
\gamma$, we further store the number of hits $n_t(C)$, and average
reward $\mu_t(C)$. After choosing an arm $x_t \in B$, we update
these stored quantities to account for the new data point, and
recompute the dependent quantities $I_t(B)$, $W_t(B)$ and
$r_t(B)$.

In the event that we change the active set $\mathcal A_t$, any
newly-stored quantities can be computed directly from the past data
points $(x_s, Y_s)$, $s \le t$. With this procedure, we can then
show that our algorithm runs in almost linear time.

\begin{theorem}
\label{thm:complexity}
On the event $E_T$, the computational complexity of
Algorithm~\ref{alg:atb} is:
\begin{enumerate}[(ii)]
\item in the setting of
Theorem~\ref{thm:regret},
\[
\mathrm{O} \bigl(\gamma^{-(1+\log_2(p))}T \log(T) \bigr),
\]
uniformly in $\gamma$ and $\varepsilon$; and
\item in the setting of
Corollary~\ref{cor:doubling},
\[
\mathrm{O} \bigl(T \log(T)^{2 + \log_2(p)} \bigr),
\]
uniformly in $\varepsilon$.
\end{enumerate}
\end{theorem}

Finally, we note that together with Theorem~\ref{thm:zooming}, we can then
deduce the second part of Corollary~\ref{cor:zoom-rate}, our result
establishing computational efficiency in continuum-armed bandits.


\section*{Acknowledgements}
We would like to thank Richard Nickl, Alexandra Carpentier, and the
anonymous referees for their valuable comments and suggestions, and
EPSRC for their support under Grant EP/K000993/1.

\begin{supplement}
\stitle{Supplement to ``Adaptive-treed bandits''.}
\slink[doi]{10.3150/14-BEJ644SUPP} 
\sdatatype{.pdf}
\sfilename{BEJ644\_supp.pdf}
\sdescription{We provide proofs of our results.}
\end{supplement}

%

\printhistory
\end{document}